\documentclass{amsproc}
\usepackage{epsfig,amscd,amssymb,amsmath,amsfonts}
\usepackage[margin=1.05in]{geometry}
\usepackage[colorlinks]{hyperref}
\usepackage{enumitem}
\usepackage{xcolor}
\newtheorem{theorem}{Theorem}[section]
\newtheorem{lemma}[theorem]{Lemma}
\newtheorem{proposition}{Proposition}[section]
\theoremstyle{definition}
\newtheorem{definition}[theorem]{Definition}

\newtheorem{example}[theorem]{Example}

\theoremstyle{remark}
\newtheorem{remark}[theorem]{Remark}

\numberwithin{equation}{section}

\begin{document}

\title{Conditional Probability Matrix and the $S^2$-rank}

\author{Mihai D. Staic}
\address{Department of Mathematics and Statistics, Bowling Green State University, Bowling Green, OH 43403 }
\address{Institute of Mathematics of the Romanian Academy, PO.BOX 1-764, RO-70700 Bu\-cha\-rest, Romania.}
\email{mstaic@bgsu.edu}







\begin{abstract} Using the $det^{S^2}$ map from  \cite{sta2}, we introduce the notion of $S^2$-rank of a matrix of type  $d\times \frac{s(s-1)}{2}$.  As an application, we show that the  conditional probability matrix associated to two random variables has the $S^2$-rank equal to $1$. Under suitable conditions we prove that the converse of this result also holds.  
\end{abstract}

\maketitle

\section{Introduction} 

The rank of a matrix is an important invariant for several problems in mathematics.   In probability theory the rank can be used to determine when two discrete random variables are independent. More precisely, if $X$ and $Y$ are discrete random variables, and $A=(a_{i,j})$ is the joint probabilities matrix determined by $a_{i,j}=P(X=i,Y=j)$ for all 
$1\leq i\leq n$  and $1\leq j\leq m$,  then $X$ and $Y$ are independent if and only if  the rank of the $n\times m$ matrix $A$  is equal to $1$ (see \cite{dss}). 

The map $det^{S^2}:V_2^6\to k$ (for a vector space $V_2$ of dimension $2$) was introduced in \cite{sta2} as a natural generalization of the determinant map. It was obtain from an exterior algebra-like construction inspired by work of Pirashvili \cite{p} and Voronov \cite{vo}.  Properties of this map and generalizations were studied in \cite{lss} and \cite{sv}. In particular, there exists a geometrical interpretation of the condition $det^{S^2}((v_{i,j})_{1\leq i<j\leq 4})=0$ that we will use in this paper (see Theorem \ref{th1}).

Using the $det^{S^2}$ map, we we introduce the notion of $S^2$-rank of a matrix of type  $d\times \frac{s(s-1)}{2}$. We investigate those matrices that have the $S^2$-rank equal to $1$, and give an application to probability theory. More precisely, if $X:D\to \{1,2,\dots, s\}$,  and $Y:D\to \{1,2,\dots, d\}$ are random variables, and  $B=(b_{i,j}^k)$ is the conditional probability matrix (i.e. $b_{i,j}^k=P(Y=k\vert i<X\leq j)$ for $1\leq i<j\leq s$ and $1\leq k\leq d$), then the $S^2$-rank of $B$ is $1$. We show that under some mild conditions the converse of this statement is also true.  We  give a few examples, and discuss possible applications to statistics. 

\section{Preliminaries}

\subsection{The $det^{S^2}$ map}
In this paper $k$ is a field with $char(k)=0$, $V_d=k^d$ is a $k$-vector space of dimension $d$, and $\mathcal{B}_d=\{e_1,e_2,\dots e_d\}$ is the standard basis in $V_d$. For applications we take $k=\mathbb{R}$. 

Consider $v_{i,j}\in V_d$ for all $1\leq i <j\leq s$, we denote by 
$$\mathfrak{V}=(v_{i,j})_{1\leq i<j\leq s}\in V_d^{\frac{s(s-1)}{2}},$$ 
the collection of $\frac{s(s-1)}{2}$ vectors $v_{i,j}$. Using the standard basis $\mathcal{B}_d$, we can identify $\mathfrak{V}$ with a $d\times \frac{s(s-1)}{2}$ matrix whose columns are  determined by the vectors $v_{i,j}=\sum_{k=1}^dv_{i,j}^ke_k$. The rows of this matrix are indexed by elements in the set $\{1,2,\dots, d\}$, while columns are indexed by pairs $(i,j)$ where $1\leq i<j\leq s$. 
For example, if $d=2$, $s=4$ and $v_{i,j}=\begin{bmatrix}
\alpha_{i,j}\\
\beta_{i,j}
\end{bmatrix}$, then we identify $\mathfrak{V}=(v_{i,j})_{1\leq i<j\leq 4}\in V_2^6$ with the $2\times 6$ matrix
$$\begin{bmatrix}
\alpha_{1,2} & \alpha_{2,3} & \alpha_{3,4} & \alpha_{1,3} & \alpha_{2,4} &\alpha_{1,4}\\
\beta_{1,2} & \beta_{2,3} & \beta_{3,4} &\beta_{1,3} & \beta_{2,4}&\beta_{1,4}\\
\end{bmatrix}.$$
Notice that the order of the columns has to be fixed (we cannot permute them). To keep track  of columns and operations that are allowed on this matrix see the tensor upper triangular notation and the results from  \cite{lss}, \cite{sta2} and \cite{sv}. In this paper we do not need that much detail, so we will use this simpler matrix notation.

We recall from \cite{sta2} the formula for the map $det^{S^2}:V_2^6\to k$. For $v_{i,j}=\begin{bmatrix}
\alpha_{i,j}\\
\beta_{i,j}
\end{bmatrix}\in k^2$ we have
\begin{eqnarray*} &det^{S^2}\begin{bmatrix}
\alpha_{1,2} & \alpha_{2,3} & \alpha_{3,4} & \alpha_{1,3} & \alpha_{2,4} &\alpha_{1,4}\\
\beta_{1,2} & \beta_{2,3} & \beta_{3,4} &\beta_{1,3} & \beta_{2,4}&\beta_{1,4}\\
\end{bmatrix}=&\\
&\alpha_{1,2}\alpha_{2,3}\alpha_{3,4}\beta_{1,3}\beta_{2,4}\beta_{1,4}+
\alpha_{1,2}\beta_{2,3}\alpha_{3,4}\beta_{1,3}\beta_{2,4}\alpha_{1,4}+
\alpha_{1,2}\beta_{2,3}\beta_{3,4}\alpha_{1,3}\alpha_{2,4}\beta_{1,4}&\\
&+\beta_{1,2}\beta_{2,3}\alpha_{3,4}\alpha_{1,3}\alpha_{2,4}\beta_{1,4}+
\beta_{1,2}\alpha_{2,3}\beta_{3,4}\beta_{1,3}\alpha_{2,4}\alpha_{1,4}+
\beta_{1,2}\alpha_{2,3}\beta_{3,4}\alpha_{1,3}\beta_{2,4}\alpha_{1,4}&\\
&-\beta_{1,2}\beta_{2,3}\beta_{3,4}\alpha_{1,3}\alpha_{2,4}\alpha_{1,4}-
\beta_{1,2}\alpha_{2,3}\beta_{3,4}\alpha_{1,3}\alpha_{2,4}\beta_{1,4}-
\beta_{1,2}\alpha_{2,3}\alpha_{3,4}\beta_{1,3}\beta_{2,4}\alpha_{1,4}&\\
&-\alpha_{1,2}\alpha_{2,3}\beta_{3,4}\beta_{1,3}\beta_{2,4}\alpha_{1,4}-
\alpha_{1,2}\beta_{2,3}\alpha_{3,4}\alpha_{1,3}\beta_{2,4}\beta_{1,4}-
\alpha_{1,2}\beta_{2,3}\alpha_{3,4}\beta_{1,3}\alpha_{2,4}\beta_{1,4}.&
\end{eqnarray*}
\begin{remark}
This formula was obtained from an exterior algebra-like construction. It was proved in \cite{sta2} that  the map $det^{S^2}$ is a the unique nontrivial multilinear map defined on $V_2^6$ which has the property that $det^{S^2}((v_{i,j})_{1\leq i<j\leq 4})=0$ if there exists $1\leq x<y<z\leq 4$ such that $v_{x,y}=v_{x,z}=v_{y,z}$.
\end{remark}
 Alternatively, one can show that 
\begin{eqnarray}
det^{S^2}((v_{i,j})_{1\leq i<j\leq 4})=det\begin{bmatrix}
\alpha_{1,2} & \alpha_{2,3} & 0 & -\alpha_{1,3} & 0 &0\\
\beta_{1,2} & \beta_{2,3} & 0 &-\beta_{1,3} & 0&0\\
\alpha_{1,2} & 0 & 0 & 0 & \alpha_{2,4} & -\alpha_{1,4}\\
\beta_{1,2} & 0 & 0 & 0 & \beta_{2,4} & -\beta_{1,4}\\
0 & 0 & \alpha_{3,4} & \alpha_{1,3} & 0 &-\alpha_{1,4}\\
0 & 0 & \beta_{3,4} & \beta_{1,3} & 0 &-\beta_{1,4}
\end{bmatrix},
\label{eqdS2}
\end{eqnarray}
where $det$ is the usual determinant map (see \cite{sv}).

\begin{theorem} (\cite{sv}) Take $V_2=\mathbb{R}^2$ and let $(v_{i,j})_{1\leq i<j\leq 4}\in V_2^6$. Then the following are equivalent
\begin{enumerate}
\item $det^{S^2}((v_{i,j})_{1\leq i<j\leq 4})=0$.\\
\item There exist four points $Q_1$, $Q_2$, $Q_3$, $Q_4$ in the plane $\mathbb{R}^2$ and $\lambda_{i,j}\in\mathbb{R}$ for  $1\leq i<j\leq 4$ not all zero such that $\lambda_{i,j}v_{i,j}=\overrightarrow{Q_iQ_j}$. 
\end{enumerate}
\label{th1}
\end{theorem}

\begin{remark} For a vector space $V_d$ of dimension $d$, there exists a map $det^{S^2}:V_d^{d(2d-1)}\to k$  that is nontrivial, multilinear, and has the property that $det^{S^2}((v_{i,j})_{1\leq i<j\leq 2d})=0$ if there exist $1\leq x<y<z\leq 2d$ such that $v_{x,y}=v_{x,z}=v_{y,z}$ (see \cite{m4}). In this paper we will only use the case $d=2$, so we do not give details about the general case. For $d=2$, and $d=3$ it is know that a map with the above property is unique up to a constant (see \cite{lss} and \cite{sta2}). For $d>3$ uniqueness is still an open question. 
\end{remark}

\subsection{Independent variables}
To put the results from this paper in context, we recall from \cite{dss} a theorem about independent variables and the joint probability matrix. 

Let $P$ be a probability function on a space $D$, and take $X:D\to \{1,2,...,n\}$ and $Y:D\to \{1,2,...,m\}$ to be two discrete random variables. We say that $X$ and $Y$ are independent if $$P(X=i,Y=j)=P(X=i)P(Y=j),$$
for all $1\leq i\leq n$, and $1\leq j\leq m$ (see \cite{lm}). Define the joint probability matrix of $X$ and $Y$ as the $n\times m$ matrix with entries 
$$a_{i,j}=P(X=i, Y=j),$$
for all $1\leq i\leq n$, and all $1\leq j\leq m$. One has the following result (see \cite{dss}). 
\begin{proposition}
The two random variables $X$ and $Y$ are independent if and only if the joint probability matrix  $(a_{i,j})_{i,j}$ has rank $1$. 
\end{proposition}

\section{Conditional Probability Matrix}
In this section we discuss  the notion of $S^2$-rank and show that the conditional probability matrix has the  $S^2$-rank equal to $1$.

\begin{definition} Let $\mathfrak{V}=(v_{i,j})_{1\leq i<j\leq s}\in V_d^{\frac{s(s-1)}{2}}$ be a nonzero element such that $\displaystyle{v_{i,j}=\sum_{k=1}^dv_{i,j}^ke_k}$. For every $1\leq x_1<x_2<x_3<x_4\leq s$, and every $1\leq a_1<a_2\leq q$ we define the $2$-minor $M_{x_1,x_2,x_3,x_4}^{a_1,a_2}(\mathfrak{V})$ as the element $(w_{i,j})_{1\leq i<j\leq 6}\in V_2^6$ determined by 
\begin{eqnarray}
w_{i,j}=\begin{bmatrix}
v_{x_i,x_j}^{a_1}\\
v_{x_i,x_j}^{a_2}
\end{bmatrix}.
\end{eqnarray}
 We say that $\mathfrak{V}$ has the $S^2$-rank equal to $1$ if  for every $1\leq x_1<x_2<x_3<x_4\leq s$, and every $1\leq a_1<a_2\leq d$  we have 
$det^{S^2}(M_{x_1,x_2,x_3,x_4}^{a_1,a_2}(\mathfrak{V}))=0$. 
\label{rankS2}
\end{definition}

\begin{remark} As mentioned above, for every $q$ there exists a map $det^{S^2}_{q}:V_q^{q(2q-1)}\to k$ that generalizes the determinant map. 
One can easily extend Definition \ref{rankS2} by saying that $\mathfrak{V}=(v_{i,j})_{1\leq i<j\leq s}\in V_d^{\frac{s(s-1)}{2}}$ has $S^2$-rank equal to $q-1$ if there exists a $(q-1)$-minor of $\mathfrak{V}$ such that 
$det^{S^2}_{q-1}(M_{y_1,\dots,y_{2q-2}}^{b_1,\dots,b_{q-1}}(\mathfrak{V}))\neq 0$, and for every $q$-minor  of $\mathfrak{V}$ we have $det^{S^2}_{q}(M_{x_1,\dots,x_{2q}}^{a_1,\dots,a_{q}}(\mathfrak{V}))=0$. Since we are only interested in the case $q=2$ we will not elaborate on the general definition. 
\end{remark}

\begin{definition}
Let $P$ be a probability function on $D$. Take $X:D\to \{1,2,\dots,s\}$, and $Y:D\to \{1,2,\dots,d\}$ two discrete random variable, such that $X(\delta)>1$ for all $\delta\in D$, and $P(i<X\leq j)>0$ for all $1\leq i<j\leq s$. 
\begin{enumerate}
\item The conditional probability matrix is the $d\times \frac{s(s-1)}{2}$ matrix determined by 
$$v_{i,j}^a=P(Y=a ~\vert ~ i<X\leq j)=\frac{P(Y=a,~i<X\leq j)}{P(i<X\leq j)},$$ 
for all $1\leq a\leq d$, and $1\leq i<j\leq s$ .
\item The distribution vectors $p_i\in V_d$ for $1\leq i\leq s$, are determined by $$p_i^a=P(Y=a,X\leq i),$$ 
for all $1\leq a\leq d$ and $1\leq i\leq s$. 
\item The  distribution weights $\lambda_{i,j}\in [0,1]$ are defined by   $$\lambda_{i,j}=P(i<X\leq j),$$ for all $1\leq i<j\leq s$. 
\end{enumerate}
\end{definition}

Notice that the conditional probability matrix  can be identified with  $\mathfrak{V}_{X,Y}=(v_{i,j})_{1\leq i<j\leq s}\in V_d^{\frac{s(s-1)}{2}}$, where $v_{i,j}=\sum_{a=1}^dv_{i,j}^ae_a\in V_d$. Also, because $X(\delta)>1$ for all $\delta\in D$ we have that $p_1=0\in V_d$.

\begin{theorem} Let $X:D\to \{1,2,\dots,s\}$, and $Y:D\to \{1,2,\dots,d\}$  be  discrete random variable such that $X(\delta)>1$ for all $\delta\in D$, and $P(i<X\leq j)>0$ for all $1\leq i<j\leq s$. Consider the conditional probability matrix $\mathfrak{V}_{X,Y}=(v_{i,j})_{1\leq i<j\leq s}\in V_d^{\frac{s(s-1)}{2}}$,  the distribution vectors $p_i\in V_d$ for all $1\leq i\leq s$, and the distribution weights $\lambda_{i,j}$ for all $1\leq i<j\leq s$ defined as above. Then we have 
\begin{enumerate}
\item For all $1\leq i<j\leq s$, and  $1\leq a\leq d$ we have $\lambda_{i,j}\in (0,1]$, $v_{i,j}^a\in [0,1]$, ${\displaystyle \sum_{b=1}^{d}v_{i,j}^b=1}$, and $$\lambda_{i,j}v_{i,j}=(p_j-p_i).$$ 
\item For all $1\leq i< j< k\leq s$ there exists $\alpha_{i,j,k}\in (0,1)$ such that $$v_{i,k}=\alpha_{i,j,k}v_{i,j}+(1-\alpha_{i,j,k})v_{j,k},$$ in particular  $rank[v_{i,j},v_{i,k},v_{j,k}]\leq 2$.
\item The $S^2$-rank of the conditional probability matrix $\mathfrak{V}_{X,Y}=(v_{i,j})_{1\leq i<j\leq s}\in V_d^{\frac{s(s-1)}{2}}$ is $1$. 
\end{enumerate}
\label{th2}
\end{theorem}
\begin{proof} The first statement follows directly from the definitions of $v_{i,j}$, $\lambda_{i,j}$ and $p_i$, and the fact that $P(Y=a,i<X\leq k)=P(Y=a,i<X\leq j)+P(Y=a,j<X\leq k)$. 

For the second statement notice that  if $1\leq i<j<k\leq s$ then
\begin{eqnarray}\label{eql1}
\lambda_{i,j}v_{i,j}-\lambda_{i,k}v_{i,k}+\lambda_{j,k}v_{j,k}=0, 
\end{eqnarray} so we can take $\alpha_{i,j,k}=\frac{\lambda_{i,j}}{\lambda_{i,k}}$, and use the fact that $\lambda_{i,k}=\lambda_{i,j}+\lambda_{j,k}$. The last statement is a consequence of Theorem \ref{th1} (it also follows directly from  equations \ref{eqdS2} and \ref{eql1}). 
\end{proof}

\begin{example} \label{example1} Suppose that we want to analyze if students in a certain class watched the Super Bowl halftime show.  
For this purpose we ask two of the students in that class to calculate  conditional probabilities $P(Y=1|i<X\leq j)$ for the variables $X$ and $Y$ defined below.  We take $D$ to be the set of all students in the class, and $Y:D\to \{1,2\}$,  
$Y=1$ if the student watched the show, and $2$ otherwise.  Finally, $X:D\to \{1,2,3,4\}$ with the convention that $X=2$ if the student is a freshman or sophomore,  $X=3$ if the student is a junior or senior, and $X=4$ if the student is a graduate student. 
The two students submit the tables from Figure \ref{studentA} and  Figure \ref{studentB}.   
\begin{figure}[htbp]
 \begin{tabular}{|c|c|c|c|c|c|c|c|c|}
  \hline
 $(i,j)$&$(1,2)$&$(2,3)$&$(3,4)$&$(1,3)$&$(2,4)$&$(1,4)$\\
  \hline
 $P(Y=1|i<X\leq j)$ &0.5&0.8&0.2&0.7&0.7&0.6\\
    \hline
    \end{tabular}
		\caption{Table from Student A \label{studentA}}
\end{figure}

\begin{figure}[htbp]
 \begin{tabular}{|c|c|c|c|c|c|c|c|c|}
  \hline
$(i,j)$&$(1,2)$&$(2,3)$&$(3,4)$&$(1,3)$&$(2,4)$&$(1,4)$\\
  \hline
 $P(Y=1|i<X\leq j)$ &0.5&0.75&0.25&0.7&0.65&0.625\\
    \hline
    \end{tabular}
		\caption{Table from Student B \label{studentB}}
\end{figure}

Assuming that one of them is right, we want to decide which one has done the correct computation, and what is the minimum number of students in that class.  

We consider the  conditional probability matrix corresponding  to tables submitted by students A and B, respectively 
$$\mathfrak{V}_A=\begin{bmatrix}
0.5 & 0.8 & 0.2& 0.7 & 0.7&0.6\\
0.5 & 0.2& 0.8 &0.3 & 0.3&0.4\\
\end{bmatrix},$$ 
and
$$\mathfrak{V}_B=\begin{bmatrix}
0.5 & 0.75 & 0.25& 0.7 & 0.65&0.625\\
0.5 & 0.25& 0.75 &0.3 & 0.35&0.375\\
\end{bmatrix}.$$  One can compute $det^{S^2}(\mathfrak{V}_A)=-0.007\neq 0$, and $det^{S^2}(\mathfrak{V}_B)=0$. From Theorem \ref{th2} we know that the table submitted by student A is wrong.

Next, consider the system associated to $\mathfrak{V}_B$
\begin{eqnarray*}
 \begin{cases}
\lambda_{1,2}v_{1,2}-\lambda_{1,3}v_{1,3}+\lambda_{2,3}v_{2,3}=0\\
\lambda_{1,2}v_{1,2}-\lambda_{1,4}v_{1,4}+\lambda_{2,4}v_{2,4}=0\\
\lambda_{1,3}v_{1,3}-\lambda_{1,4}v_{1,4}+\lambda_{3,4}v_{3,4}=0,
 \end{cases}
\end{eqnarray*}
or equivalently 
$$
\begin{bmatrix}
0.5 & 0.75 & 0 & -0.7 & 0 &0\\
0.5 & 0.25 & 0 &-0.3 & 0&0\\
0.5 & 0 & 0 & 0 & 0.65 & -0.625\\
0.5 & 0 & 0 & 0 & 0.35 & -0.375\\
0 & 0 & 0.25 & 0.7 & 0 &-0.625\\
0 & 0 & 0.75& 0.3 & 0 &-0.375
\end{bmatrix} \begin{bmatrix}
\lambda_{1,2}\\
\lambda_{2,3}\\
\lambda_{3,4}\\
\lambda_{1,3}\\
\lambda_{2,4}\\
\lambda_{1,4}
\end{bmatrix}=
\begin{bmatrix}
0\\
0\\
0\\
0\\
0\\
0
\end{bmatrix}.
$$
After we solve it, we get the general solution
\begin{eqnarray}
\begin{bmatrix}
\lambda_{1,2}\\
\lambda_{2,3}\\
\lambda_{3,4}\\
\lambda_{1,3}\\
\lambda_{2,4}\\
\lambda_{1,4}
\end{bmatrix}=\lambda
\begin{bmatrix}
1\\
4\\
1\\
5\\
5\\
6
\end{bmatrix}.\label{eqlam}
\end{eqnarray}
The smallest  solution for which $\lambda_{i,j}v_{i,j}\in \mathbb{Z}_+^2$ for all $1\leq i<j\leq 4$ is when $\lambda=4$, so the minimum class size is $24$.  In that case, the distribution is described in the table from  Figure \ref{studentB2}.  
\begin{figure}[htbp]
 \begin{tabular}{|c|c|c|c|c|c|c|c|c|}
  \hline
  &  Fr. or Soph.   &Jr. or Snr.& Grad. & UG & Jr, Snr. or Gr. &All\\
  \hline
Watched the show &2&12&1&14&13&15\\
    \hline
Did not watch the show &2&4&3&6&7&9\\
    \hline
    \end{tabular}
		\caption{Distribution table for $\mathfrak{V}_B$\label{studentB2}}
\end{figure}
\end{example}

\section{Main Result}

In this section we prove that under suitable conditions the converse of Theorem \ref{th2} holds.  First, we recall a trivial linear algebra result that will be used several times in this section. 
\begin{remark} Let $v_1,v_2,v_3\in \mathbb{R}^d$ such that  $rank[v_1,v_2,v_3]=2$. Then  the vector equation 
\begin{eqnarray}x_1v_1+
x_2v_2+x_3v_3=0,\label{eq2}
\end{eqnarray}  
has a nontrivial solution that is unique up to multiplication with a constant.
In particular, if $(\lambda_1,\lambda_2,\lambda_3)$ and $(\mu_1,\mu_2,\mu_3)$ are solutions for the equation \ref{eq2} such that $\lambda_1=\mu_1\neq 0$, then $\lambda_2=\mu_2$, and $\lambda_3=\mu_3$. 
\label{remark1}
\end{remark}

We have the following result that generalizes Theorem \ref{th1}. 
\begin{lemma} Let $v_{i,j}=\begin{bmatrix}
\alpha_{i,j}^1\\
\alpha_{i,j}^2\\
\vdots\\
\alpha_{i,j}^d
\end{bmatrix}\in \mathbb{R}^d$ for all $1\leq i<j\leq s$. 
Suppose that 
\begin{enumerate}
\item For every $1\leq i<j<k\leq s$, there exist $a_{i,j,k}> 0$, $b_{i,j,k}> 0$, and $c_{i,j,k}> 0$ such that 
\begin{eqnarray}
a_{i,j,k}v_{i,j}-b_{i,j,k}v_{i,k}+c_{i,j,k}v_{j,k}=0.
\end{eqnarray}
\item For every $1\leq i<j<k<l\leq s$, and all $1\leq a\leq d$ there exists $b$ (that depends on $a$, $i$, $j$, $k$ and $l$) such that 
\begin{eqnarray}rank\begin{bmatrix}
\alpha_{i_1,i_2}^a&\alpha_{i_1,i_3}^a&\alpha_{i_2,i_3}^a\\
\alpha_{i_1,i_2}^b&\alpha_{i_1,i_3}^b&\alpha_{i_2,i_3}^b
\end{bmatrix}=2,
\end{eqnarray}
for all $i_1<i_2<i_3\in \{i,j,k,l\}$. 
\item The $S^2$-rank of $(v_{i,j})_{1\leq i<j\leq s}\in V_d^{\frac{s(s-1)}{2}}$ is $1$.
\end{enumerate}
Then there exist $p_i\in \mathbb{R}^d$ for $1\leq i\leq s$, and $\lambda_{i,j}>0$ for $1\leq i<j\leq s$ such that 
\begin{eqnarray}\lambda_{i,j}v_{i,j}=p_j-p_i,
\end{eqnarray}
for all  $1\leq i<j\leq s$. 
\label{lemma1}
\end{lemma}
\begin{proof} 
We will prove this result by induction.  When $s=4$ and  $d=2$ this follows from Theorem \ref{th1}. Indeed, since $d=2$, and the $S^2$-rank of $(v_{i,j})_{1\leq i<j\leq 4}$ is $1$ (i.e. $v_{i,j}\in \mathbb{R}^2$ and $det^{S^2}((v_{i,j})_{1\leq i<j\leq 4})=0$), we know from Theorem \ref{th1} that there exist $p_1$, $p_2$, $p_3$, $p_4\in \mathbb{R}^2$, and $\lambda_{i,j}\in \mathbb{R}$ not all zero, such that $$\lambda_{i,j}v_{i,j}=p_j-p_i,$$ for all  $1\leq i<j\leq 4$. We still need to show that $\lambda_{i,j}> 0$ for all $1\leq i<j\leq 4$. 

Since not all $\lambda_{i,j}$ are zero, let's assume that $\lambda_{1,2}\neq 0$ (the other cases are similar). If necessary, after multiplying all $\lambda_{i,j}$ with $-1$,  we may assume that $\lambda_{1,2}> 0$. Notice that
$(\lambda_{1,2},\lambda_{1,3}, \lambda_{2,3})$ and $(a_{1,2,3},b_{1,2,3},c_{1,2,3})$ are nontrivial solutions for the equation 
$$x_{1,2}v_{1,2}-x_{1,3}v_{1,3}+x_{2,3}v_{2,3}=0.$$ 
Since $rank[v_{1,2},v_{1,3},v_{2,3}]=2$, it follows from Remark \ref{remark1} that $(\lambda_{1,2},\lambda_{1,3}, \lambda_{2,3})=c(a_{1,2,3},b_{1,2,3},c_{1,2,3})$ for some constant $0\neq c\in \mathbb{R}$. Finally, because $a_{1,2,3}>0$, $b_{1,2,3}>0$, $c_{1,2,3}>0$, and $\lambda_{1,2}>0$, it follows that $c>0$, and so $\lambda_{1,3}>0 $ and $\lambda_{2,3}> 0$. 

This argument can be extended to show that $\lambda_{i,j}> 0$ for all $1\leq i<j\leq 4$. More precisely, 
using the fact that $\lambda_{1,2}> 0$ and  the linear dependence relation among $v_{1,2}$, $v_{1,4}$ and $v_{2,4}$,  one gets that $\lambda_{1,4}> 0$ and $\lambda_{2,4}> 0$. Then, using the fact that $\lambda_{2,3}> 0$, and the linear dependence relation among $v_{2,3}$, $v_{2,4}$ and $v_{3,4}$, one gets that $\lambda_{3,4}> 0$, which proves our statement for the case $s=4$ and $d=2$.

First we will take $s=4$ and do induction over $d$. Notice that from the case $d=2$ we know that for every $1\leq a\leq d$ there exists $1\leq b\leq d$,  $\lambda_{i,j}^{a,b}> 0$, and $p_i^{a,b}\in  \mathbb{R}^2$ such that 
$$\lambda_{i,j}^{a,b}\begin{bmatrix}
\alpha_{i,j}^a\\
\alpha_{i,j}^b
\end{bmatrix}=p_j^{a,b}-p_{i}^{a,b},$$
for all $1\leq i<j\leq 4$. 
We need to show that $\lambda_{i,j}^{a,b}$ does not depend of $(a,b)$, and that we can glue together the vectors $p_i^{a,b}$ to get the statement we want.

 After permuting the elements of $\{1,2,\dots, d\}$ we may assume that we solved the problem for the set $\{1,2,\dots,d-1\}$. More precisely, if  $w_{i,j}=\begin{bmatrix}
\alpha_{i,j}^1\\
\alpha_{i,j}^2\\
\vdots\\
\alpha_{i,j}^{d-1}
\end{bmatrix}\in \mathbb{R}^{d-1}$ there exist $\lambda_{i,j}> 0$, and $p_i=\begin{bmatrix}
p_i^1\\
p_i^2\\
\vdots\\
p_i^{d-1}
\end{bmatrix}\in \mathbb{R}^{d-1}$ such that 
$\lambda_{i,j}v_{i,j}=p_j-p_i,$
for all $1\leq i<j\leq 4$. Notice that for all $1\leq i<j<k\leq 4$ we have that $(\lambda_{i,j}, \lambda_{i,k},\lambda_{j,k})$, and $(a_{i,j,k}, b_{i,j,k}, c_{i,j,k})$ are nontrivial solutions of the vector equation
$$x_{i,j}w_{i,j}-x_{i,k}w_{i,k}+x_{j,k}w_{j,k}=0.$$ 
Since the $rank([w_{i,j},w_{i,k},w_{j,k}])=2$, from Remark \ref{remark1} we know that  $(\lambda_{i,j}, \lambda_{i,k},\lambda_{j,k})$ is a multiple of  $(a_{i,j,k}, b_{i,j,k}, c_{i,j,k})$.

Let $a=d$, then there exist $b\in \{1,2,\dots,d-1\}$ such that if $u_{i,j}=\begin{bmatrix}
v_{i,j}^a\\
v_{i,j}^b\\
\end{bmatrix}$ then $rank([u_{i,j},u_{i,k},u_{j,k}])=2$ for all $1\leq i<j<k\leq 4$. From the case $d=2$ we know that there exist $\mu_{i,j}> 0$ for all $1\leq i<j\leq 4$, and $q_{i}=\begin{bmatrix}
q_i^a\\
q_i^b\\
\end{bmatrix}\in \mathbb{R}^2$  for all $1\leq i\leq 4$, such that 
$$\mu_{i,j}u_{i,j}=q_j-q_i$$ for all $1\leq i<j\leq 4$. 

For all $1\leq i<j\leq 4$ we have that $(\mu_{i,j},\mu_{i,k},\mu_{j,k})$ and $(a_{i,j,k}, b_{i,j,k}, c_{i,j,k})$ are nontrivial solutions of the vector equation
$$\mu_{i,j}u_{i,j}-\mu_{i,k}u_{i,k}+\mu_{j,k}u_{j,k}=0.$$
Since the $rank([u_{i,j},u_{i,k},u_{j,k}])=2$ from Remark \ref{remark1} we have that  $(\mu_{i,j}, \mu_{i,k},\mu_{j,k})$ is a nonzero multiple of  $(a_{i,j,k}, b_{i,j,k}, c_{i,j,k})$. 
 
After rescaling $(\mu_{i,j})_{1\leq i<j\leq 4}$ we may assume that $\mu_{1,2}=\lambda_{1,2}>0$. Since $(\mu_{1,2}, \mu_{1,3},\mu_{2,3})$ and  $(\lambda_{1,2}, \lambda_{1,3},\lambda_{2,3})$ are both nonzero multiple of $(a_{i,j,k}, b_{i,j,k}, c_{i,j,k})$,  we get that $\mu_{1,3}=\lambda_{1,3}$ and $\mu_{2,3}=\lambda_{2,3}$. Similarly, using the linear dependence relation among $w_{1,2}$, $w_{1,4}$ and $w_{2,4}$, we get that  $\mu_{1,4}=\lambda_{1,4}$ and $\mu_{2,4}=\lambda_{2,4}$. Finally, using the linear dependence relation among $w_{2,3}$, $w_{2,4}$ and $w_{3,4}$ (and the fact that now we know  $\mu_{2,3}=\lambda_{2,3}$), we get that  $\mu_{3,4}=\lambda_{3,4}$.
This shows that,  if we take 
 $\widetilde{p_i}=\begin{bmatrix}
p_i^1\\
p_i^2\\
\vdots\\
p_i^{d-1}\\
q_i^{d}
\end{bmatrix}\in \mathbb{R}^d$
 we have 
$$\lambda_{i,j}v_{i,j}=\widetilde{p_j}-\widetilde{p_i},$$
for all $1\leq i<j\leq 4$. And so, by induction, we proved our statement when $s=4$. 

Next we will do induction over $s$. We just checked the case $s=4$, so we take $s>4$.  Assume that for all $1\leq i< j\leq s-1$ there exist $\lambda_{i,j}>0$, and for all 
$1\leq i\leq s-1$ there exist $p_i\in \mathbb{R}^d$ such that 
\begin{eqnarray}
\lambda_{i,j}v_{i,j}=p_j-p_i,
\end{eqnarray} 
for all $1\leq i<j\leq s-1$. We need to construct $\lambda_{i,s}>0$ for all $1\leq i\leq s-1$ and $p_{s}\in \mathbb{R}^d$. 

For any $3\leq k\leq s-1$, from the case $s=4$ applied to the set $\{1,2,k,s\}\subseteq \{1,2,\dots, s\}$ we know that there exist $\mu_{i,j}^{(k)}>0$ for all $i<j\in  \{1,2,k,s\}$ and $q_i^{(k)}\in \mathbb{R}^d$ for all $i\in \{1,2,k,s\}$ such that 
$$\mu_{i,j}^{(k)}v_{i,j}=q_j^{(k)}-q_i^{(k)},$$
for all $i<j\in \{1,2,k,s\}$. By replacing $q_i^{(k)}$ with $q_i^{(k)}+p_1-q_1^{(k)}$, we may assume that $q_1^{(k)}=p_1$.

After rescaling $\mu_{i,j}^{(k)}$ we may assume that $\mu_{1,2}^{(k)}=\lambda_{1,2}$. Notice that $(\lambda_{1,2},\lambda_{1,k}, \lambda_{2,k})$, and $(\mu_{1,2}^{(k)},\mu_{1,k}^{(k)},\mu_{2,k}^{(k)})$ are nontrivial solutions for the vector equation 
$$x_{1,2}v_{1,2}-x_{1,k}v_{1,k}+x_{2,k}v_{2,k}.$$ 
Since $rank[v_{1,2},v_{1,k},v_{2,k}]=2$, it follows from Remark \ref{remark1} that $\lambda_{1,k}=\mu_{1,k}^{(k)}$ and $\lambda_{2,k}=\mu_{2,k}^{(k)}$ for all $3\leq k\leq s-1$. 
In particular, we have $$\lambda_{1,2}v_{1,2}=p_2-p_1=q_2^{(k)}-p_1,$$ 
and 
$$\lambda_{1,k}v_{1,k}=p_k-p_1=q_k^{(k)}-p_1,$$
 which implies that $q_2^{(k)}=p_2$ and $q_k^{(k)}=p_k$ for all $3\leq k\leq s-1$. 

Notice that for all $3\leq k<l\leq s-1$ we have that $(\mu_{1,2}^{(k)},\mu_{1,s}^{(k)},\mu_{2,s}^{(k)})$ and $(\mu_{1,2}^{(l)},\mu_{1,s}^{(l)},\mu_{2,s}^{(l)})$ are nontrivial solutions of the vector equation 
$$x_{1,2}v_{1,2}-x_{1,s}v_{1,s}+x_{2,s}v_{2,s}=0.$$
Since $\mu_{1,2}^{(k)}=\mu_{1,2}^{(l)}=\lambda_{1,2}$, and $rank[v_{1,2},v_{1,s},v_{2,s}]=2$  we have that 
$\mu_{1,s}^{(k)}=\mu_{1,s}^{(l)}$ and $\mu_{2,s}^{(k)}=\mu_{2,s}^{(l)}$ for all $3\leq k<l\leq s-1$. So we can denote these constants by $\lambda_{1,s}$ and  $\lambda_{2,s}$ respectively. We have 
$$\lambda_{1,s}v_{1,s}=q_{s}^{(k)}-q_{1}^{(k)}=q_{s}^{(l)}-q_{1}^{(l)}$$
for all $3\leq k<l\leq s-1$. Since $q_{1}^{(k)}=q_{1}^{(l)}=p_1$, we get that $q_{s}^{(k)}=q_{s}^{(l)}$ for all $3\leq k<l\leq s-1$, so we can denote this vector by $p_{s}$ and we get 
\begin{eqnarray}
\lambda_{1,s}v_{1,s}=p_{s}-p_1,
\end{eqnarray}
\begin{eqnarray}
\lambda_{2,s}v_{2,s}=p_{s}-p_2.
\end{eqnarray}

Finally for all $3\leq k\leq s-1$ we define 
$$\lambda_{k,s}=\mu_{k,s}^{(k)}.$$ 
We know that $$\mu_{1,k}^{(k)}v_{1,k}-\mu_{1,s}^{(k)}v_{1,s}+\mu_{k,s}^{(k)}v_{k,s}=0,$$ and so since $\lambda_{1,k}=\mu_{1,k}^{(k)}$,  and $\lambda_{1,s}=\mu_{1,s}^{(k)}$ we get
$$(p_k-p_1)-(p_{s}-p_1)+\lambda_{k,s}v_{k,s}=0$$ or in other words 
\begin{eqnarray}
\lambda_{k,s}v_{k,s}=p_{s}-p_k.
\end{eqnarray}
And so, by induction we proved our statement.  
\end{proof}

\begin{remark} Condition $a_{i,j,k}>0$, $b_{i,j,k}>0$ and $c_{i,j,k}>0$ is not necessary for the proof of this lemma. One can replace it with  $a_{i,j,k}\neq 0$, $b_{i,j,k}\neq 0$ and $c_{i,j,k}\neq 0$, and get a result where $\lambda_{i,j}\neq 0$. However, in the next theorem we need the result with $\lambda_{i,j}>0$.

\end{remark}

We have the following converse to Theorem \ref{th2}. 
\begin{theorem}
Let  $v_{i,j}=\begin{bmatrix}
\alpha_{i,j}^1\\
\alpha_{i,j}^2\\
\vdots\\
\alpha_{i,j}^d
\end{bmatrix}\in \mathbb{R}^d$ for all $1\leq i<j\leq s$.  Assume that 
\begin{enumerate}
\item For all $1\leq i<j\leq s$, and all $1\leq a\leq d$ we have $v_{i,j}^a\in [0,1]$, and $\sum_{a=1}^{d}v_{i,j}^a=1$. 
\item For every $1\leq i<j<k\leq s$, there exist $a_{i,j,k}> 0$, $b_{i,j,k}> 0$, $c_{i,j,k}> 0$ such that 
\begin{eqnarray}
a_{i,j,k}v_{i,j}-b_{i,j,k}v_{i,k}+c_{i,j,k}v_{j,k}=0.
\end{eqnarray}
\item For every $1\leq i<j<k<l\leq s$, and all $1\leq a\leq d$ there exists $b$ (that depends on $a$, $i$, $j$, $k$ and $l$) such that 
$$rank\begin{bmatrix}
\alpha_{i_1,i_2}^a&\alpha_{i_1,i_3}^a&\alpha_{i_2,i_3}^a\\
\alpha_{i_1,i_2}^b&\alpha_{i_1,i_3}^b&\alpha_{i_2,i_3}^b
\end{bmatrix}=2.$$
for all $i_1<i_2<i_3\in \{i,j,k,l\}$.  
\item The $S^2$-rank of $(v_{i,j})_{1\leq i<j\leq s}$ is $1$.
\end{enumerate}
Then there exist two random variables $X:(0,1]\to \{1,2,\dots,s\}$, and $Y:(0,1]\to \{1,2,...,d\}$ such that 
$$v_{i,j}^a=P(Y=a\vert i<X\leq j),$$ for all $1\leq i<j\leq s$, and $1\leq a\leq d$. 
\label{th3}
\end{theorem}

\begin{proof}

From Lemma \ref{lemma1} we know that there exist $\lambda_{i,j}>0$ for all $1\leq i<j\leq s$, and $p_i\in \mathbb{R}^d$ for all $1\leq i\leq s$, such that $\lambda_{i,j}v_{i,j}=p_j-p_i$ for all $1\leq i<j\leq s$.  We can normalize the $(\lambda_{i,j})_{1\leq i<j\leq s}$ such that $\lambda_{1,s}=1$. Changing $p_i$ to $p_i-p_1$ we may assume that $p_1=0\in \mathbb{R}^d$. 

For $1\leq i<j<k\leq s$ we have $\lambda_{i,j}v_{i,j}-\lambda_{i,k}v_{i,k}+\lambda_{j,k}v_{j,k}=0$. Summing all the entries in these vectors we get 
$$\lambda_{i,j}(\sum_{a=1}^dv_{i,j}^a)-\lambda_{i,k}(\sum_{a=1}^dv_{i,k}^a)+\lambda_{j,k}(\sum_{a=1}^dv_{j,k}^a)=0.$$ Since $\sum_{a=1}^dv_{i,j}^a=1$ for all $1\leq i<j\leq s$ this means that 
$$\lambda_{i,j}-\lambda_{i,k}+\lambda_{j,k}=0,$$ in particular 
$\lambda_{i,j}=\lambda_{1,j}-\lambda_{1,i}$. Notice that if $1\leq i<j\leq s$ then $\lambda_{1,i}<\lambda_{1,j}$. 

Define $X:(0,1]\to \{1,2,\dots, s\}$, and $Y:(0,1]\to \{1,2,\dots, d\}$ determined by 
$$X(t)=k~ {\rm if} ~t\in (\lambda_{1,k-1},\lambda_{1,k}],$$ 
(here we use the convention $\lambda_{1,1}=0$), and 
$$Y(t)=h ~ {\rm if} ~t\in (\lambda_{1,i-1}+\sum_{a=1}^{h-1}(p_i^a-p_{i-1}^a),\lambda_{1,i-1}+\sum_{a=1}^{h}(p_i^a-p_{i-1}^a)]$$ for some  $1\leq i\leq s$. 

Let $P$ be the probability given by the standard measure on the interval $(0,1]$. From the above definitions we have that for every $1\leq i<j\leq s$ 
\begin{eqnarray*}
P(i<X\leq j)&=&\lambda_{1,j}-\lambda_{1,i}\\&=&\lambda_{i,j}.
\end{eqnarray*}
For every $h\in \{1,2,\dots,d\}$  we have
\begin{eqnarray*}
P(Y=h)&=&\sum_{1\leq a\leq s}(p^h_a-p_{a-1}^h)\\
&=&p_{s}^h-p_{1}^h\\
&=&p_{s}^h\\
&=&\lambda_{1,s}v_{1,s}^h\\ 
&=&v_{1,s}^h.\\ 
\end{eqnarray*}
More generally, for every $1\leq i<j\leq s$, and every $h\in \{1,2,\dots,d\}$ we have
\begin{eqnarray*}
P(Y=h,~ i<X\leq j)&=&\sum_{i< a\leq j}(p^h_a-p_{a-1}^h)\\
&=&p_j^h-p_i^h\\
&=&\lambda_{i,j}v_{i,j}^h.
\end{eqnarray*}
In particular we get  $$P(Y=h~\vert~ i<X\leq j)=v_{i,j}^h,$$
which proves our statement.

\end{proof}

\begin{remark} Let $v_{1,2}=\begin{bmatrix}
0.5\\
0.5
\end{bmatrix}$, $v_{2,3}=\begin{bmatrix}
0.8\\
0.2
\end{bmatrix}$, $v_{3,4}=\begin{bmatrix}
0.2\\
0.8
\end{bmatrix}$, $v_{1,3}=\begin{bmatrix}
0.7\\
0.3
\end{bmatrix}$, $v_{2,4}=\begin{bmatrix}
0.7\\
0.3
\end{bmatrix}$ and $v_{1,4}=\begin{bmatrix}
0.6\\
0.4
\end{bmatrix}$. One can see that if we take $\alpha_{1,2,3}=\frac{1}{3}$, $\alpha_{1,2,4}=\frac{1}{2}$, $\alpha_{1,3,4}=\frac{4}{5}$, and $\alpha_{2,3,4}=\frac{5}{6}$ then we have 
$$v_{i,k}=\alpha_{i,j,k}v_{i,j}+(1-\alpha_{i,j,k})v_{j,k},$$
for all $1\leq i<j<k\leq 4$. As it was noticed in Example \ref{example1},  $det^{S^2}((v_{i,j})_{1\leq i<j\leq 4})=-0.007\neq 0$, and so the $S^2$-rank of $(v_{i,j})_{1\leq i<j\leq 4}$ is not equal to $1$. This means that condition  four in the statement of Theorem \ref{th3} is not a consequence of the other three. 

Condition $rank[v_{i,j},v_{i,k},v_{j,k}]=2$ is not necessary. For example, if the two variables $X$ and $Y$ are independent then the conditional probability matrix will have rank equal to $1$. However, we were not able to prove a  statement without it. \label{remark2}
\end{remark}

\begin{example} Let's assume that in Example \ref{example1} we have a third student that considered the set $E$ of all undergraduate students in that class, and a  random variable $Z:E\to \{1,2,3,4\}$  defined by $Z=2$ if the student in freshman or sophomore, $Z=3$ if the student is junior, and $Z=4$ if the student is senior. He  submitted the  table from Figure \ref{studentC}. 
\begin{figure}[htbp]
 \begin{tabular}{|c|c|c|c|c|c|c|c|c|}
  \hline
$(i,j)$&$(1,2)$&$(2,3)$&$(3,4)$&$(1,3)$&$(2,4)$&$(1,4)$\\
  \hline
 $P_C(Y=1|i<Z\leq j)$ &0.5&1&0.6&0.8&0.75&0.7\\
    \hline
    \end{tabular}
		\caption{Table from Student C \label{studentC}}
\end{figure}

The corresponding matrix is
$\mathfrak{V}_C=\begin{bmatrix}
0.5 & 1& 0.6& 0.8 & 0.75&0.7\\
0.5 & 0& 0.4 &0.2 & 0.25&0.3\\
\end{bmatrix}.$
Just like in Example \ref{example1} one can find  the general  solution
\begin{equation}\begin{bmatrix}
\mu_{1,2}\\
\mu_{2,3}\\
\mu_{3,4}\\
\mu_{1,3}\\
\mu_{2,4}\\
\mu_{1,4}
\end{bmatrix}=\mu
\begin{bmatrix}
2\\
3\\
5\\
5\\
8\\
10
\end{bmatrix}.\label{eqlam2}
\end{equation}

If we denote by $w_{i,j}$ the columns of matrix $\mathfrak{V}_C$, then the smallest  value of $\mu$ for which $\mu_{i,j}w_{i,j}\in \mathbb{Z}_+^2$ for all $1\leq i<j\leq 4$ is $\mu=1$, so a minimum class size is $10$.  In that case, the distribution is described in the table from  Figure \ref{studentB2}  
\begin{figure}[htbp]
 \begin{tabular}{|c|c|c|c|c|c|c|c|c|}
  \hline
    &  Fr. or Soph.   &Jr. & Snr. & Fr. or Soph. & Jr, Snr  &UG\\
  \hline
Watched the show &1&3&3&4&6&7\\
    \hline
Did not watch the show &1&0&2&1&2&3\\
    \hline
    \end{tabular}
		\caption{Distribution table for $\mathfrak{V}_C$}
\end{figure}

Next, let's notice that Table \ref{studentB} and Table \ref{studentC} are compatible. Indeed, one can see that $P(Y=1|1<X\leq 2)=P(Y=1|1<Z\leq 2)$, $P(Y=1|2<X\leq 3)=P(Y=1|2<Z\leq 4)$ and $P(Y=1|1<X\leq 3)=P(Y=1|1<Z\leq 4)$. This allows us to define a new random variable $T:D\to \{1,2,3,4,5\}$ determined by $T=2$ if the student in freshman or sophomore, $T=3$ if the student is junior, $T=4$ if the student is senior, and $T=5$ if the student is a graduate student. Using this new variable $T$, the combined information is presented in Table \ref{studentComb}. 

\begin{figure}[htbp]
 \begin{tabular}{|c|c|c|c|c|c|c|c|c|c|c|c|c|}
  \hline
$(i,j)$&$(1,2)$&$(2,3)$&$(3,4)$&$(4,5)$&$(1,3)$&$(2,4)$&$(3,5)$&$(1,4)$&$(2,5)$&$(1,5)$\\
  \hline
 $P(Y=1|i<T\leq j)$ &0.5&1&0.6&0.25 &0.8&0.75&?&0.7&0.65&0.625\\
    \hline
    \end{tabular}
		\caption{Combined information \label{studentComb}}
\end{figure}

Take $t_{i,j}=\begin{bmatrix}
P(Y=1|i<T\leq j) \\
P(Y=2|i<T\leq j) 
\end{bmatrix}$ for all $1\leq i<j\leq 5$. 
Notice that there is missing information in the table from Figure \ref{studentComb}, we do not know the value of $P(Y=1|3<T\leq 5)$.  However we can use the other information to determine it. If the data from these two tables is compatible we can take $\mu=2$ in Equation \ref{eqlam2}, and $\lambda=4$  in Equation \ref{eqlam} to get $\mu_{1,3}=\nu_{1,3}=10$ and $\lambda_{1,4}=\nu_{1,5}=24$. The equation $\nu_{1,3}t_{1,3}-\nu_{1,5}t_{1,5}+\nu_{3,5}t_{3,5}=0$ becomes 
$$10\begin{bmatrix}
0.8 \\
0.2 
\end{bmatrix}-24\begin{bmatrix}
0.625 \\
0.375 
\end{bmatrix}+\nu_{3,5}t_{3,5}=0,$$ which implies that $\nu_{3,5}t_{3,5}=\begin{bmatrix}
7 \\
7 
\end{bmatrix}$ and so $$t_{3,5}=\begin{bmatrix}
0.5 \\
0.5 
\end{bmatrix}~~~{\rm and }~~~\nu_{3,5}=14.$$
The corresponding distribution is presented in  Table \ref{dist35}, and the minimum class size is still $24$.

\begin{figure}[htbp]
 \begin{tabular}{|c|c|c|c|c|c|c|c|c|c|c|c|c|}
  \hline
       &  Fr.    &Jr. & Snr. &  Gr  & Fr.  & Jr.   & Sr. & UG & Jr. & All \\
       &  Soph.  &    &      &      & Soph &  Snr. &Gr.  &    & Snr.&    \\
       &         &    &      &      & Jr.  &       &     &    & Gr. &    \\
  \hline
Watched the show       &2&6&6&1&8&12& 7& 14 & 13&15\\
    \hline
Did not watch the show &2&0&4&3&2&4 & 7& 6  & 7& 9\\
    \hline
    \end{tabular}
		\caption{Combined distribution table \label{dist35}}
\end{figure}

\end{example}

\begin{example} Building again on Example \ref{example1}, let's assume that after initially submitting an incorrect answer, student A  tries to get extra credit and submits the new data from Table \ref{studentA2}. Here,  $U:D\to \{1,2,3\}$ is determined by $U=1$ if the student enjoyed the show, $U=2$ if the student did not enjoy the show, and $U=3$ if the student did not watched the show. 

\begin{figure}[htbp]
 \begin{tabular}{|c|c|c|c|c|c|c|c|c|}
  \hline
$(i,j)$&$(1,2)$&$(2,3)$&$(3,4)$&$(1,3)$&$(2,4)$&$(1,4)$\\
  \hline
 $P_A(U=1|i<X\leq j)$  &0.375&0.4375 &0.125& ?  &?&?\\
  \hline
 $P_A(U=2|i<X\leq j)$  &0.125&0.3125 &0.125&?   &?&? \\
  \hline
	$P_A(U=3|i<X\leq j)$ &0.5  &0.25   &0.75 &0.3      &0.35 &0.375\\
    \hline
    \end{tabular}
		\caption{New table from Student A \label{studentA2}}
\end{figure}
Notice that the data from Table \ref{studentB} and Table \ref{studentA2} are compatible, since $P_B(Y=2|i<X\leq j)=P_A(U=3|
i<X\leq j)$ for all $1\leq i<j\leq 4$. Again, there is some missing information in Table \ref{studentA2}, however that can be recovered by noticing that Equation \ref{eqlam} is the general solution of the problem associated to the data in Table \ref{studentA2}. In particular one can check that the smallest  solution for which $\lambda_{i,j}v_{i,j}\in \mathbb{Z}_+^2$ for all $1\leq i<j\leq 4$ is when $\lambda=8$, so the minimum class size is $48$.  In that case, the distribution is described in the table from  Figure \ref{studentA3}. 
\begin{figure}[htbp]
 \begin{tabular}{|c|c|c|c|c|c|c|c|c|}
  \hline
  &  Fr. or Soph.   &Jr. or Snr.& Grad. & UG & Jr, Snr. or Gr. &All\\
  \hline
Enjoyed the show       &3&14&1&17&15&18\\
    \hline
Did not enjoy the show &1&10&1&11&11&12\\
    \hline
Did not watch the show &4&8&6&12&14&18\\
    \hline
    \end{tabular}
		\caption{Distribution table for $\mathfrak{V}_{A_{NEW}}$\label{studentA3}}
\end{figure}

\end{example} 

\begin{remark} It would be interesting to see if these results can be applied to statistics. One serious problem is that the rank, and the $S^2$-rank are highly sensitive to small variations of the vectors $v_{i,j}$. So, in order to use the ideas from this paper to explicit problems, one needs to decide what range values of  $det$, and $det^{S^2}$ are small enough to be considered $0$ in this context, and how to correct the errors (fit data) in that situation.  Since we don't have any expertise in that area, we leave this problem to statisticians to decide. 
\end{remark}



\section*{Acknowledgment}
We thank J. Chen and S. Lippold  for feedback on an earlier version of this paper. 

\bibliographystyle{amsalpha}

\end{document}